\documentclass[11pt]{amsart}
\usepackage{amsmath, amssymb, amsthm, amsfonts, mathrsfs}

\setbox0=\hbox{$+$}
\newdimen\plusheight
\plusheight=\ht0
\def\+{\;\lower\plusheight\hbox{$+$}\;}

\setbox0=\hbox{$-$}
\newdimen\minusheight
\minusheight=\ht0
\def\-{\;\lower\minusheight\hbox{$-$}\;}

\setbox0=\hbox{$\cdots$}
\newdimen\cdotsheight
\cdotsheight=\plusheight
\def\cds{\lower\cdotsheight\hbox{$\cdots$}}

\newtheorem{conjecture}{Conjecture}

\theoremstyle{definition}

\usepackage{graphicx, epsfig}
\theoremstyle{definition}


\numberwithin{equation}{section}
\theoremstyle{plain}
\newtheorem{theorem}{Theorem}[section]
\newtheorem{corollary}[theorem]{Corollary}
\newtheorem{proposition}[theorem]{Proposition}

\newtheorem{lemma}[theorem]{Lemma}

\usepackage{titlesec}

\titleformat*{\section}{\LARGE\bfseries}

\makeatletter
\renewcommand\section{\@startsection{section}{1}{\z@}%
                                  {-3.5ex \@plus -1ex \@minus -.2ex}%
                                  {2.3ex \@plus.2ex}%
                                  {\normalfont\large\bfseries}}
\makeatother
\usepackage{lipsum}

\begin{document}
\title{Hyperbolic $3$-manifolds of bounded volume and trace field degree II}
\author{BoGwang Jeon}
\maketitle
\begin{abstract}
In this paper, we prove the Bounded Height Conjecture which the author formulated in \cite{jeon}. As a corollary, it follows that there are only a finite number of hyperbolic $3$-manifolds of bounded volume and trace field degree. 
\end{abstract}

\section{Introduction} \label{Int}
In \cite{jeon}, the author formulated the following conjecture as a way to prove the conjecture that there are only a finite number of hyperbolic $3$-manifolds of bounded volume and trace field degree.

\begin{conjecture} \label{CBHC}
\emph{(Bounded Height Conjecture in Geometric Topology)} 
Let $M$ be a $k$-cusped hyperbolic $3$-manifold. Then the height of any Dehn filling point of $M$ is uniformly bounded. 
\end{conjecture}

In the same paper, the author gave affirmative answers toward it under some restrictions. For instance, the following is one of the main theorems in \cite{jeon}. 
\begin{theorem} \label{MAIN}
Suppose that the above conjecture is true for any $s$-cusped manifolds where $1\leq s\leq k-1$. Let $X$ be the holonomy variety of $k$-cusped hyperbolic $3$-manifold $M$. 
If $X$ is simple, then the above conjecture is true for $M$. 
\end{theorem}
The proof of the above theorem essentially relies on the following theorem which was proved by P. Habegger in \cite{hab}. \footnote{Actually the author used a slightly generalized version of Habegger\rq{}s theorem, which was proved by the author.}
\begin{theorem}\label{habe}
\emph{(Bounded Height Conjecture in Number Theory=Habegger's theorem)} Let $X \subset (\overline{\mathbb{Q}}^*)^n$ be an irreducible variety over $\mathbb{\overline Q}$. 
Then there is a Zariski open subset $X^{oa}$ of $X$ so that 
the height is bounded in the intersection of $X^{oa}$ with the union of algebraic subgroups of dimension $\leq n-\text{dim } X$.
\end{theorem}

In fact, Habegger's theorem already tells us a lot about the uniform boundedness of heights of most Dehn filling points unless $X^{oa}\neq \emptyset$. In Theorem \ref{MAIN}, the holonomy variety $X$ being \lq\lq{}simple\rq\rq{} is an ideal assumption on $X$ so that each subvariety $X$, $(X\backslash X^{oa})$ and $(X\backslash X^{oa})\big\backslash(X\backslash X^{oa})^{oa}$ (and so on) contains only a finite number of anomalous subvarieties. 
As a result, we prove the conjecture by applying Habegger\rq{}s theorem repeatedly, a finite number of times, to each of them and their anomalous subvarieties.  

Although the holonomy variety being \lq\lq{}simple\rq\rq{} is defined from a purely algebro-geometrical (or number-theoretical) viewpoint, interestingly enough, it turns out that this definition gives a very nice structure from the hyperbolic geometric side as well, as the following theorem shows:


\begin{theorem}\cite{jeon}\label{main3}
\footnote{The author proved an analogous version of Theorem \ref{main3} for any $k$-cusped manifold ($k\geq 2$) in his unpublished work.}  
Let $M$ be a $2$-cusped hyperbolic $3$-manifold with rationally independent cusp shapes. 
If the holonomy variety of $X$ is not simple, then the two cusps of $M$ are strongly geometrically isolated.  
\end{theorem}

Note that, for the $2$-cusped case, the holonomy variety $X$ being \lq\lq{}non-simple\rq\rq{} is equivalent to $X^{oa}=\emptyset$.

If a hyperbolic $3$-manifold $M$ has strongly geometrically isolated cusps, then its holonomy variety $X$ is always non-simple 
so there's no obvious way to get the desired result from Habegger's theorem. 
However, in this case, it is known that its structure appears as a product of two less cusped manifolds (see Section 4.1 in \cite{jeon}), so we still get uniform boundedness of the heights of Dehn filling points by the induction step.

In some sense, this fact implies that Habegger's theorem is not optimal in our context. Generally, Habegger's theorem deals with an arbitrary variety and all algebraic subgroups, which intersect it, satisfying certain dimension conditions. 
But on the other hand, in our case, we are only interested in a small neighborhood of a specific point of $X$ (i.e. the point corresponding to the complete structure) and a fixed type of algebraic subgroup intersecting it. 
Intuitively, these observations give us some impression that, by loosening conditions of Habbeger's original proof, there might be some possibility to get a refined version of it, which would be more suitable to our context. 

After all, it turns out that, using some nice properties of the holonomy variety, we can get the following modified form of Habegger\rq{}s theorem by slightly modifying its proof: \footnote{See the first two paragraphs of the proof of Theorem 1.3 in \cite{jeon} to see why Theorem \ref{BHC1} implies the affirmative answer to Conjecture \ref{MAIN}.}


\begin{theorem}\label{BHC1}
Let $X$ be the holonomy variety of a $k$-cusped hyperbolic $3$-manifold. Then there exists a small neighborhood (in the sense of classical topology) of the point corresponding to the complete structure so that the height of its intersection point with any algebraic subgroup defined by Dehn filling equations is uniformly bounded. \footnote{One of the major differences between Theorem \ref{BHC1} and Habegger's theorem is that, in the first case, 
we can get uniformly bounded height without even removing any of its anomalous subvarieties.} 
\end{theorem}

By the work in \cite{jeon}, this theorem implies the following corollary:

\begin{corollary}
Let $M$ be a $k$-cusped hyperbolic $3$-manifold. Then there are a finite number of hyperbolic Dehn fillings of $M$ of bounded trace field degree. 
\end{corollary}

By the Jorgensen-Thurston theory (Theorem 2.5 in \cite{jeon}), the above corollary implies the  following:

\begin{corollary}
There are only a finite number of hyperbolic $3$-manifolds of bounded volume and trace field degree. 
\end{corollary}

Here\rq{}s outline of the paper. In Section \ref{pro}, we review some important properties of $X$ and prove the key lemmas which will be used in the proof of Theorem \ref{BHC1}. Applying these properties and lemmas to Habegger\rq{}s original proof, we prove Theorem \ref{BHC1} in Section \ref{proof}. In particular, the last section is independent from the rest, and it is readable without any background on hyperbolic geometry. We provide the sufficient information about $X$ from the number theoretical viewpoint. 

\section{Key Lemma} \label{pro} 

We first collect some needed properties of the holonomy variety $X$ of a given $k$-cusped hyperbolic $3$-manifold $M$. Throughout the paper, the holonomy variety $X$ is the one defined by (5.2) in \cite{jeon}. So $X$ is a $k$-dimensional algebraic variety embedded in 
\begin{equation}\label{coordinates}
\mathbb{C}^{2k+n} \big(:=(M_1,L_1,\dots,M_k,L_k,z_1,\dots,z_n)\big) 
\end{equation}
and the point corresponding to the complete structure is a smooth point of $X$.

To simply the notation let $c$ be the point corresponding to the complete structure. By taking the logarithm on each coordinate, we can get a $k$-dimensional complex manifold \big(say $\text{Def}^*(M)$\big) which is biholomorphic to a small neighborhood (in the sense of classical topology) of $c$ in $X$ \big(say $N$\big). Furthermore, by projecting $\text{Def}^*(M)$ onto the first $2k$-coordinates, we get another biholomorphic $k$-dimensional complex manifold \big(say $\text{Def}(M)$\big) which is embedded in $\mathbb{C}^{2k}$. In other words, we have the following compositions of biholomorphic maps:
\begin{gather*}
N \xrightarrow{\text{log}} \text{Def}^*(M)\xrightarrow{\text{Proj}_{2k}}\text{Def}(M)
\end{gather*}
where \lq\lq{}$\text{log}$\rq\rq{} is the $(2k+n)$-fold product of the usual logarithmic map and \lq\lq{}$\text{Proj}_{2k}$\rq\rq{} represents the projection on the first $2k$-coordinates. 

As in \cite{jeon}, let 
\begin{gather*}
  u_i:=\log M_i\\
  v_i:=\log L_i 
\end{gather*}
for each $1\leq i\leq k$. Then the following facts hold:
\\
\\
(i) $\text{Proj}_{2k}\big(\log (c)\big)=(0,\dots0)\in\mathbb{C}^{2k}$\\
(ii) Using $u_i$ ($1\leq i\leq k$) as coordinates, each $v_i$ can be represented as a holomorphic function of $u_1,\dots, u_k$, and is of the following form
\begin{equation*}
v_i=u_i\cdot\tau_i(u_1,\dots,u_k)
\end{equation*}
where $\tau _i(u_1,\dots,u_k)$ is an even function of its arguments with $\tau_i(0,\dots,0)\notin \mathbb{R}$ for each $i$.\\

Using (i) and (ii), we prove the following lemma:
\begin{lemma} \label{special}
Let $\varphi_0$ be a $k\times (2k+n)$ matrix of rank $k$ such that, for each $i^{\text{th}}$ row ($1\leq i \leq k$), all the entries, possibly except for the ones in the $(2i-1)^{\text{th}}$ and $(2i)^{\text{th}}$ columns, are all equal to zero. Let $a_i$ and $b_i$ be the entries of the $(2i-1)^{\text{th}}$ and $(2i)^{\text{th}}$ columns for each $i^{\text{th}}$ row. \footnote{Note that at least one of $a_i$ and $b_i$ is nonzero for each $i$.}(See \eqref{matrix} in Section \ref{proof} for this matrix type.) Then the differential 
\begin{equation*}
d_z(\varphi_0|_{\text{Def}^*(M)}):T_z\text{Def}^*(M)\rightarrow\mathbb{C}^k
\end{equation*}
is isomorphic where $z=\log (c)\in \text{Def}^*(M)$. \footnote{Here we consider $\varphi_0$ as a linear map from $\mathbb{C}^{2k+n}$ to $\mathbb{C}^{k}$ and $\varphi_0|_{\text{Def}^*(M)}$ a holomorphic map between the two $k$-dimensional complex manifolds $\text{Def}^*(M)$ and $\mathbb{C}^{k}$.}   
\end{lemma}
\begin{proof}
Since $\text{Def}^*(M)$ is biholomorphic to $(u_1,\dots,u_k)$, it is enough to prove that the rank of the Jacobian of the following linear map\\ 
\begin{equation*}
\varphi'_0:\quad (u_1,\dots,u_k)\longrightarrow(a_1u_1+b_1v_1,\dots,a_ku_k+b_kv_k)
 \end{equation*}\\
is equal to $k$ at $(0,\dots,0)$. The Jacobian of $\varphi'_0$ at $(0,\dots,0)$ is\\
\[ \left( \begin{array}{cccc}\label{mat1}
a_1+b_1\tau_1 & 0 & \cdots & 0 \\
0 &  a_2+b_2\tau_2& \cdots & 0 \\
\vdots & \vdots & \ddots & \vdots \\
0 & 0 & \cdots & a_k+b_k\tau_k \end{array} \right)\]\\ 
where $\tau_i:=\tau_i(0,\dots,0)$. Since $a_i,b_i\in \mathbb{R}, \tau_i\notin\mathbb{R}$, and at least one of $a_i$ and $b_i$ is nonzero, it is a matrix of rank $k$. 
\end{proof}

Since $c$ is a smooth point of $X$, the following lemma easily follows from the above lemma: 

\begin{lemma} \label{special 1}
Let $\varphi_0$ be the same matrix given in Lemma \ref{special} and $Y$ be an irreducible subvariety of $X$ containing $c$. Then there exists a small open neighborhood $N$ (in the sense of classical topology) of some smooth point in $Y$ such that the complex manifold $N'$, which is obtained by taking the logarithm on each coordinate of $N$, satisfies the following: 
\footnote{If $c$ is a smooth point in $Y$ as well, then we can pick $z$ as $\log (c)$. But, if it is singular, then we have to take some other point near $\log(c)$.}
\begin{equation*}
\text{The differential }d_z(\varphi_0|_{N'}):T_z N'\rightarrow\mathbb{C}^k \text{ is injective for some }z\in N'. 
\end{equation*}
\end{lemma}

\section{Proof of Theorem \ref{BHC1}}\label{proof}

In this section, we prove Theorem \ref{BHC1}. Before we start, we introduce some notation which will be frequently used throughout the section. First, we identify $G^{2k+n}_m$ with the non-vanishing coordinates in the affine $(2k+n)$-space $\overline {\mathbb{Q}}^{2k+n}$ or $\mathbb{C}^{2k+n}$ (i.e.~$G^{2k+n}_m=(\overline{\mathbb{Q}}^*)^{2k+n}$ or $(\mathbb{C^*})^{2k+n}$). We will also denote the coordinates of $G^{2k+n}_m$ by 
\begin{equation*}
(M_1,L_1,\dots,M_k,L_k,z_1,\dots,z_n),
\end{equation*}
which is the same as in the previous section.

The crucial facts which enable us to modify Habegger\rq{}s original proof are the following:\\
\\
(i) $c$ is a smooth point of $X$.\\
(ii) $X$ satisfies Lemma \ref{special 1} in Section \ref{pro}.\\
(iii) We are only interested in the intersection of a small neighborhood (in the sense of classical topology) of $c$ and an algebraic subgroup of fixed type (i.e. the one defined by \eqref{Dehn1}) .\\

In Lemma 6.3 of \cite{hab}, there is an important dichotomy where the first case (i.e. Lemma 6.3 (i) in \cite{hab}) 
falls into the case of anomalous subvariety 
and the second case (i.e. Lemma 6.3 (ii) in \cite{hab}) is used to get uniformly bounded height. In our case, thanks to Lemma \ref{special 1} (which is exactly the negation of the first case), we are always in the second case and thus this will allow us to get the desired result without even removing any of its anomalous subvarieties. 

Now we may rewrite Theorem \ref{BHC1} as follows:
\begin{theorem}\label{BHC3}
Let $X$ be a $k$-dimensional variety in $G^{2k+n}_m(\overline{\mathbb{Q}})$ satisfying Lemma \ref{special 1} in Section \ref{pro} for some smooth point $c\in X$. Then there exists a small neighborhood $\Sigma$ (in the sense of classical topology) of $c$ so that the height of its intersection with any algebraic subgroup defined by the following type of equations 
\begin{equation}\label{Dehn1}
\begin{split}
M_1^{p_1}L_1^{q_1}=1\\
\vdots \qquad \\
M_k^{p_k}L_k^{q_k}=1
\end{split}
\end{equation}
is uniformly bounded.
\end{theorem}

Let $Dehn^{[k]}$ denote the set of all algebraic subgroups defined by equations of the form \eqref{Dehn1} in $G^{2k+n}_m$. And let $\text{Mat}^*_{k(2k+n)}$ be the set of all $k\times(2k+n)$ matrices such that, for each $i^{\text{th}}$ row ($1\leq i \leq k$), all the entries, possibly except for the ones in $(2i-1)^{\text{th}}$ and $(2i)^{\text{th}}$ columns, are all equal to zero.  In other words, $\text{Mat}^*_{k(2k+n)}$ is the set of $k\times(2k+n)$ matrices of of the following type:\\
\begin{equation}\label{matrix}
\begin{pmatrix}
a_1 & b_1 &  0     & \cdots & \cdots & \cdots & \cdots & \cdots & \cdots & 0 \\
0     &   0   & a_2  &    b_2   &   0       & \cdots & \cdots & \cdots & \cdots & 0 \\
\vdots & \vdots & \vdots & \vdots &\vdots&\vdots& \vdots & \vdots & \ddots & \vdots\\
0     &   0   & \cdots & \cdots & \cdots & a_k &  b_k & 0 & \cdots & 0
\end{pmatrix}.
\end{equation}\\

The reason for considering only this type of matrix is obvious. Since we are dealing with an algebraic subgroup of fixed form (i.e. the one defined by \eqref{Dehn1}),  we only need a matrix representing its indices, which is clearly contained in $\text{Mat}^*_{k(2k+n)}$. 

Sometimes when we mention $\text{Mat}_{k(2k+n)}$, it simply means the set of all $k\times (2k+n)$ matrices. All other notation \big(such as $\Delta$ or $\mathcal{C}(S,\epsilon)$ for some $S\subset G^{2k+n}_m(\overline{\mathbb{Q}})$\big) is exactly the same as in \cite{hab} unless otherwise stated. 
\begin{lemma} \label{lemma 3}
Let $Q>1$ be a real number and let $\varphi_0\in \emph{Mat}^*_{k(2k+n)}(\mathbb{R})$, there exist $q\in \mathbb{Z}$ and $\varphi\in\emph{Mat}^*_{k(2k+n)}(\mathbb{Z})$ such that 
\begin{equation*}
1\leq q\leq Q \text{ and } |q\varphi_0-\varphi|\leq\frac{\sqrt{k(2k+n)}}{Q^{1/(k(2k+n))}}.
\end{equation*}
\end{lemma}
\begin{proof}
This follows from Lemma 4.1 in \cite{hab}.
\end{proof}

We define $\mathcal{K}^*_{k(2k+n)}\subset \text{Mat}^*_{k(2k+n)}(\mathbb{R})$ to be the compact set of all matrices whose rows are orthonormal. So all elements of $\mathcal{K}^*_{k(2k+n)}$ have rank $k$. \footnote{By the definition, all the rows of a matrix in $\text{Mat}^*_{k(2k+n)}$ are naturally orthogonal.}
\begin{lemma}\label{lemma 4}
Suppose $W\subset \emph{Mat}^*_{k(2k+n)}(\mathbb{R})$ is an open neighborhood of $\mathcal{K}^*_{k(2k+n)}$. 
Then there is $Q_0\geq 1$ (which may depend on $W$) with the following property. 
For $Q>Q_0$ a real number and $\varphi_0\in \emph{Mat}^*_{k(2k+n)}(\mathbb{R})$ with rank $k$, 
there exist $q\in \mathbb{Z}, \varphi\in \emph{Mat}^*_{k(2k+n)}(\mathbb{Z})$, and $\theta\in \emph{Mat}_k(\mathbb{Q})$ such that
\begin{equation}
1\leq q\leq Q,\;\; \frac{\varphi}{q}\in W,\;\; |q\theta \varphi_0-\varphi|\leq \frac{\sqrt{k(2k+n)}}{Q^{1/(k(2k+n))}}, \text{ and }|\varphi|\leq (k+1)q. 
\end{equation} 
\end{lemma}
\begin{proof}
Following the proof of Lemma 4.2 in \cite{hab}, we can see that, in our case, we can pick $\varphi$ (which is essentially obtained by the previous lemma) as an element in $\emph{Mat}^*_{k(2k+n)}(\mathbb{Z})$. The rest immediately follows from the same lemma.  
\end{proof}

\begin{lemma}\label{lemma 5}
Let $\varphi:G^{2k+n}_m\rightarrow G^k_m$ and $p\in G^{2k+n}_m(\overline{\mathbb{Q}})$, then
\begin{equation*}
h(\varphi(p))\leq \sqrt{k(2k+n)}|\varphi|h(p).
\end{equation*}
\end{lemma} 
\begin{proof}This follows from Lemma 4.3 in \cite{hab}.
\end{proof}

\begin{lemma}\label{lemma 6}
Suppose $W\subset \text{Mat}^*_{k(2k+n)}(\mathbb{R})$ is an open neighborhood of $\mathcal{K}^*_{k(2k+n)}$. 
Let $Q_0$ be the constant from Lemma \ref{lemma 4} and let $Q>Q_0$ be a real number. 
If $p\in Dehn^{[k]}$, then there exist $q\in\mathbb{Z}$ and $\varphi\in \text{Mat}^*_{k(2k+n)}(\mathbb{Z})$ such that
\begin{equation*}
1\leq q\leq Q,\;\; \frac{\varphi}{q}\in W,\;\; h(\varphi(p))\leq \frac{k(2k+n)}{Q^{1/(k(2k+n))}}h(p), \text{ and }|\varphi|\leq(k+1)q. 
\end{equation*} 
\end{lemma}
\begin{proof}
Following the proof of Lemma 4.4 in \cite{hab}, we can pick $\varphi$ as an element of $\text{Mat}^*_{k(2k+n)}(\mathbb{Z})$ in this case as well. The rest immediately follows from the same lemma.
\end{proof}

From now on, $Y(\subset X \subset G^{2k+n}_m)$ denotes an irreducible closed subvariety \lq\lq{}containing $c$\rq\rq{} and having dimension $r\geq 1$. 


\begin{proposition}\label{proposition 1}
Suppose $\varphi:G^{2k+n}_m\rightarrow G^r_m$ is a nontrivial 
homomorphism of algebraic subgroups. There exist a dense Zariski open subset $U\subset Y$ and a constant $C_7$ such that 
\begin{equation}\label{(13)}
h(\varphi(p))\geq\frac{r}{2C_1}|\varphi|\frac{\Delta_Y(\varphi)}{|\varphi|^r}h(p)-C_7
\end{equation} 
for all $p\in U(\overline{\mathbb{Q}})$ where $C_1=(4(2k+n))^r\emph{deg}(Y)$.
\end{proposition}
\begin{proof}
This follows from Proposition 5.2 in \cite{hab}.
\end{proof}

\begin{lemma}\label{lemma 8}
Let $N$ be an open neighborhood (in the sense of classical topology) of some smooth point in $Y$ and $N\rq{}$ be the complex manifold obtained by taking the logarithm on each coordinate of $N$. We further assume that $d_{z_0}(\varphi_0|_{N\rq{}})$ is an isomorphism of $\mathbb{C}$-vector spaces for some $\varphi_0\in \emph{Mat}_{r(2k+n)}(\mathbb{C})$ and $z\in N\rq{}$. Then there exist $C_8>0$ and an open neighborhood $W\subset \emph{Mat}_{r(2k+n)}(\mathbb{R})$ of $\varphi_0$ such that 
\begin{equation*}
\Delta_Y(\varphi)\geq C_8
\end{equation*} 
for all $\varphi\in W\cap\emph{Mat}_{r(2k+n)}(\mathbb{Q})$. 
\end{lemma}
\begin{proof}
It follows from Lemma 6.2 in \cite{hab}.
\end{proof}

As we mentioned earlier, by Lemma \ref{special 1}, we are always in the second case of Lemma 6.3 in \cite{hab}. That is, the following lemma holds.   
\begin{lemma}\label{lemma 9}
There exists $C_9>0$ and an open neighborhood $W\subset \emph{Mat}^*_{k(2k+n)}(\mathbb{R})$ of $\mathcal{K}^*_{k(2k+n)}$ such that for each 
$\varphi\in W\cap \emph{Mat}^*_{k(2k+n)}(\mathbb{Q})$ there is $\pi\in\prod_{rk}$ with $\Delta_Y(\pi\varphi)\geq C_9$.  
\end{lemma}
\begin{proof}
By Lemma \ref{special 1}, for any $\varphi_0\in\mathcal{K}^*_{k(2k+n)}$, there exists $z\in N\rq{}$ and $\pi\in\prod_{rk}$ such that   
$d_z(\varphi_0|_{N\rq{}})$ is an isomorphism of $\mathbb{C}$-vector spaces. By Lemma \ref{lemma 8}, we can find an open neighborhood of $\pi\varphi_0\in \text{Mat}_{r(2k+n)}(\mathbb{R})$ with the stated properties. It follows that we may find $W\rq{}_{\varphi_0}$, an open neighborhood of $\varphi_0$ in $\text{Mat}_{k(2k+n)}(\mathbb{R})$, and $C_{\varphi_0}>0$ with $\Delta_Y(\pi\varphi)\geq C_{\varphi_0}$ for all $\varphi \in W\rq{}_{\varphi_0}\cap\text{Mat}_{k(2k+n)}(\mathbb{Q})$.

Then open cover $\cup_{\varphi_0\in \mathcal{K}} W\rq{}_{\varphi_0}$ contains $\mathcal{K}^*_{k(2k+n)}$. 
Since $\mathcal{K}^*_{k(2k+n)}$ is compact, we may pass to a finite subcover and conclude that there exist $C_9$ and an open subset $W\rq{}$ of $\text{Mat}_{k(2k+n)}(\mathbb{R})$ such that for each $\varphi\in W\rq{}\cap \text{Mat}_{k(2k+n)}(\mathbb{Q})$ there is $\pi \in \prod_{rk}$ with $\Delta_Y(\pi\varphi)\geq C_9$. By taking the restriction $W=W\rq{}\cap \text{Mat}^*_{k(2k+n)}(\mathbb{R})$, we get a desired open set in $\text{Mat}^*_{k(2k+n)}(\mathbb{R})$ having the same properties. 
\end{proof}

In the proof of the next lemma, we simply follow the proof of Lemma 8.1 in \cite{hab}. 
Before proceeding, we remark the following observation which is given in \cite{hab} as well: say $\epsilon$ is an small number satisfying $0<\epsilon \leq \frac{1}{2(2k+n)}$ with 
$p\in \mathcal{C}(Dehn^{[k]},\epsilon)$, so there is $a\in Dehn^{[k]}$ and $b\in G^{2k+n}_m(\mathbb{\overline{Q}})$ with $h(b)\leq\epsilon (1+h(a))$. 
Then $h(a)=h(pb^{-1})\leq  h(p)+h(b^{-1})\leq h(p)+(2k+n)h(b) \leq h(p) +(1+h(a))/2$ by the elementary properties of height. We easily deduce 
\begin{equation}\label{(24)}
h(a)\leq 1+2h(p), \quad h(b)\leq 2\epsilon (1+h(p)).
\end{equation}

\begin{lemma} \label{lemma 11} 
There exists $\epsilon >0$ and $U\subset Y$ which is Zariski open and dense such that 
the height is bounded on $U(\overline{\mathbb{Q}})\cap \mathcal{C}(Dehn^{[k]},\epsilon)$.
\end{lemma}
\begin{proof}
By Lemma \ref{lemma 9}, there exists an open set $W\subset \text{Mat}^*_{k(2k+n)}(\mathbb{R})$ containing $\mathcal{K}^*_{k(2k+n)}$ and $C_{10}>0$ such that, for each 
$\varphi\in W\cap \text{Mat}^*_{k(2k+n)}(\mathbb{Q})$, there is $\pi\in \prod_{\text{dim}Y,k}$ with $\Delta_Y(\pi\varphi)\geq C_{10}$. 

Suppose that $Q_0$ is as in Lemma \ref{lemma 6} and $Q>Q_0$ is a fixed parameter which depends only on $X$ and $Y$. We will see later how to choose $Q$ properly. 

Let $\Theta$ denote the set of all matrices $\varphi\in \text{Mat}^*_{k(2k+n)}(\mathbb{Z})$ such that there exists an integer $q$ with $1\leq q\leq Q, \varphi/q\in W$, 
and $|\varphi|\leq (k+1)q$ (cf. Lemma \ref{lemma 6}). Clearly, $\Theta$ is a finite set. 

For each $\varphi\in \Theta$, there is a $\pi\in \prod_{\text{dim}Y,k}$ satisfying $\Delta_Y(\varphi'/q)\geq C_{10}$ with $\varphi'=\pi\varphi$. 
In particular, $\varphi'\neq 0$ since $C_{10}>0$. By homogeneity we have 
\begin{equation*}
\Delta_Y(\varphi')=q^{\text{dim}Y}\Delta_Y(\varphi'/q)\geq C_{10}q^{\text{dim}Y}.
\end{equation*} 
Now $\varphi'\neq 0$ implies $|\varphi'|\geq 1$, so we obtain the following lower bound for the factor in front of $h(p)$ in \eqref{(13)}
\begin{equation}\label{(111)}
C_{11}|\varphi'|\frac{\Delta_Y(\varphi')}{|\varphi'|^{\text{dim}Y}}\geq C_{10}C_{11}|\varphi'|\frac{q^{\text{dim}Y}}{ |\varphi'|^{\text{dim}Y}}\geq C_{10}C_{11}\frac{q^{\text{dim}Y}}{|\varphi'|^{\text{dim}Y}}
\end{equation}
where
\begin{equation*}
C_{11}=\frac{\text{dim}Y}{2(4(2k+n))^{\text{dim}Y}\text{deg}(Y)}>0.
\end{equation*}
Since $|\varphi'|\leq|\varphi|\leq(k+1)q$, \eqref{(111)} implies 
\begin{equation*}
C_{11}|\varphi'|\frac{\Delta_Y(\varphi')}{|\varphi'|^{\text{dim}Y}}\geq \frac{C_{10}C_{11}}{(k+1)^{\text{dim}Y}}.
\end{equation*} 
We denote this last quantity by $C_{12}$; it is positive and independent of $Q$ and $\varphi$.

Fix $Q$ and $\epsilon$ as below: 
\begin{equation}\label{(25)}
\begin{split}
&Q=\text{max}\left\{Q_0+1,(8k(2k+n)C_{12}^{-1})^{k(2k+n)}\right\}>Q_0,\\
&\epsilon=\text{min}\left\{\frac{1}{2(2k+n)}, \frac{\sqrt{k(2k+n)}}{k+1}\frac{1}{Q^{1+1/(k(2k+n))}}\right\}\in \left(0,\frac{1}{2(2k+n)}\right].
\end{split}
\end{equation}
Let $U_{\varphi}$ be the dense Zariski open subset of $Y$ supplied by Proposition \ref{proposition 1} applied to $\varphi$. Then the intersection
\begin{equation*}
U=\bigcap_{\varphi\in\Theta}U_{\varphi}
\end{equation*}
is a dense Zariski open subset of $Y$ (as $\Theta$ is a finite set) and 
\begin{equation}\label{(26)}
h(\varphi'(p))\geq C_{12}h(p)-C(Q)
\end{equation}
for all $p\in U(\overline{\mathbb{Q}})$ and all $\varphi\in \Theta$; here $C(Q)$ depends neither on $p$ nor on $\varphi$ (but possibly on $Q$). 

Now we assume $p\in U(\overline{\mathbb{Q}})\cap \mathcal{C}(Dehn^{[k]},\epsilon)$. 
That is, there are $a\in Dehn^{[k]}$ and $b\in G^{2k+n}_m(\overline{\mathbb{Q}})$ with $p=ab$ and $h(b)\leq \epsilon (1+h(a))$. 
Then, by Lemma \ref{lemma 6}, there exists $\varphi \in \Theta$ with 
\begin{equation*}
h(\varphi(a))\leq k(2k+n)Q^{-1/(k(2k+n))}h(a)
\end{equation*}
and so
\begin{equation}\label{(27)}
h(\varphi(a))\leq 2k(2k+n) Q^{-1/(k(2k+n))}(1+h(p))
\end{equation}
by \eqref{(24)}. 

By Lemma \ref{lemma 5}, we get 
\begin{equation*}
h(\varphi(b))\leq \sqrt{k(2k+n)}|\varphi|h(b), 
\end{equation*}
and 
\begin{equation*} 
h(\varphi(b))\leq 2\epsilon\sqrt{k(2k+n)}|\varphi|(1+h(p))
\end{equation*} 
by \eqref{(24)}. But $|\varphi|\leq(k+1)q\leq(k+1)Q$, thus 
\begin{equation}\label{(28)}
h(\varphi(b))\leq 2\epsilon\sqrt{k(2k+n)}(k+1)Q(1+h(p)).
\end{equation}
Using \eqref{(27)}, \eqref{(28)}, and elementary properties of height give
\begin{align*}
h(\varphi(p))&=h(\varphi(ab))\leq h(\varphi(a))+h(\varphi(b))\\ 
             &\leq \big(2k(2k+n)Q^{-1/(k(2k+n))}+2\epsilon\sqrt{k(2k+n)}(k+1)Q\big)(1+h(p)).
\end{align*}
The choice of $\epsilon$ made   in \eqref{(25)} implies 
\begin{equation*}
h(\varphi (p)) \leq 4k(2k+n) Q^{-1/(k(2k+n))}(1+h(p)),
\end{equation*} 
and the choice of $Q$ gives 
\begin{equation*}
h(\varphi(p))\leq C_{12}(1+h(p))/2.
\end{equation*} 
Furthermore, we have $h(\varphi'(p))\leq h(\varphi(p))$, hence
\begin{equation}\label{(29)}
h(\varphi'(p))\leq \frac{C_{12}}{2}(1+h(p)).
\end{equation}
Comparing \eqref{(26)} and \eqref{(29)}, we immediately get the desired result $h(p)\leq 1+2C_{12}^{-1}C(Q)$.
\end{proof}
Let $\Sigma$ be a small open neighborhood (in the sense of classical topology) of $c$ in $X$. Applying Lemma \ref{lemma 11} with $X=Y$ shows that there exists a nonempty Zariski open subset $U\subset X$ such that $U(\overline{\mathbb{Q}})\cap Dehn^{[k]}$ has bounded height. 
If $\Sigma$ is contained in $U$, then we are done. Otherwise, the following simple descent argument shows how to deal with the points in $(\Sigma \backslash (\Sigma \cap U(\overline{\mathbb{Q}})))\cap Dehn^{[k]}$: 
\begin{lemma}\label{lemma 12}
Suppose that there is a proper subset $S \subsetneq \Sigma$ and an $\epsilon>0$ with the following properties:\\
\\
(i) $S$ intersects with any small neighborhood (in the sense of classical topology) of $c$ but contains none of them.\\
(ii) The height is bounded from above on $S\cap \mathcal{C}(Dehn^{[k]},\epsilon)$.\\
\\
Then there exists a subset $S'\subset \Sigma$ containing $S$ with $\overline{\Sigma\backslash S'}\subsetneq\overline{\Sigma\backslash S}$ and an $\epsilon\rq{}>0$ such that the height is bounded from above on $S'\cap \mathcal{C}(Dehn^{[k]},\epsilon\rq{})$.
\end{lemma}
\begin{proof}
By the assumption on $S$ and the fact that $\Sigma\backslash S$ is non-empty, it follows that its Zariski closure $\overline{\Sigma \backslash S}$ contains $c$. 
Let $\overline{\Sigma \backslash S}=Y \cup Z$ where $Y$ an irreducible variety containing $c$, $Z$ Zariski closed and $Y \not\subset Z$. If $Y$ has positive dimension, then we may apply Lemma \ref{lemma 11} and find a dense Zariski open $U\subset Y$ and $\epsilon\rq{}>0$ such that the height is bounded from above on $U(\overline{\mathbb{Q}})\cap \mathcal{C}(Dehn^{[k]}$,$\epsilon\rq{})$. If $Y$ is a point (meaning that $Y=c$), then the existence of $U$ and $\epsilon\rq{}$ is obvious. 

Clearly, we may assume $\epsilon\rq{}\leq \epsilon$. We set $S'=S\cup(U(\overline{\mathbb{Q}})\cap \Sigma)$. By the hypothesis and previous paragraph, the height is bounded from above on $S'\cap \mathcal{C}(Dehn^{[k]},\epsilon\rq{})$. 
Of course $\Sigma \backslash S'\subset (Y \backslash U)\cup Z$ and even $\overline{\Sigma\backslash S'}\subset (Y\backslash U)\cup Z$. 
So $\overline{\Sigma\backslash S'}=\overline{\Sigma\backslash S}$ is impossible and the lemma follows.  
\end{proof}

\begin{proof}[Proof of Theorem \ref{BHC1}.]  
Let $U$ and $\epsilon>0$ be the ones obtained from Lemma \ref{lemma 11} by applying $X=Y$ to it. As we remarked earlier, if $\Sigma \subset U$, then the theorem follows. So we suppose $\Sigma \not\subset U$, and set $S_0=U\cap \Sigma$ and $\epsilon_0=\epsilon$. By induction, let\rq{}s assume $S_{l-1}$ ($l\geq 1$) be a subset of $\Sigma$ and $\epsilon_{l-1}>0$ such that $S_{l-1}\cap \mathcal{C}(Dehn^{[k]},\epsilon_{l-1})$ has bounded height. 
If $S_{l-1}$ contains any open neighborhood of $c$, then we are done. Otherwise, we apply Lemma \ref{lemma 12} and get a subset $S_l\subset \Sigma$ and an $\epsilon_l>0$ such that the height is bounded on $S_l\cap\mathcal{C}(Dehn^{[k]},\epsilon_l)$. 

Now we obtain a chain 
\begin{equation*}
X\supset \overline{\Sigma\backslash S_0}\supsetneq \overline{\Sigma\backslash S_1}\supsetneq \cdots \supsetneq\overline{\Sigma\backslash S_l}
\end{equation*} 
But $X$, being a noetherian topological space, satisfies the descending chain condition of Zariski closed sets. In our case, this means $S_l=\Sigma$ for some $l$. This completes the proof.  
\end{proof}
\section{Acknowledgement}
I would like to thank Walter Neumann and Nathan Dunfield for their support, interest and encouragement in this project throughout the years. 
I also thank Jennifer Hom for her help in correcting the English in this paper. 

\vspace{10 mm}

Department of Mathematics\

Columbia University\

2990 Broadway, New York, NY 10027\\

\emph{Email Address}: jeon@math.columbia.edu

\end{document}